\documentclass[12pt,leqno]{amsart}
\usepackage{amsmath,amsfonts,amssymb,amsthm}
\usepackage{graphicx}
\graphicspath{ }
\usepackage{wrapfig}
\usepackage{accents}
\usepackage{caption}
\usepackage{subcaption}
\usepackage{calligra}

\numberwithin{figure}{section}
\newcommand\norm[1]{\left\lVert#1\right\rVert}

\theoremstyle{plain}
\newtheorem{theorem}{Theorem}

\newtheorem{proposition}[theorem]{Proposition}
\newtheorem{corollary}{Corollary}[theorem]
\theoremstyle{definition}
\newtheorem{definition}{Definition}[section]
\newtheorem{conjecture}{Conjecture}[section]

\theoremstyle{remark}

\usepackage{mathtools}
\begin{document}
\title[polynomial growth of subharmonic functions]{polynomial growth of subharmonic functions in a strongly symmetric Riemannian manifold}
\author[A. A. Shaikh and C. K. Mondal]{Absos Ali Shaikh$^{1*}$ and Chandan Kumar Mondal$^2$}
\maketitle

\begin{abstract}
In this article we have studied some properties of subharmonic functions in a strongly symmetric Riemannian manifold with a pole. As a generalization of polynomial growth of a function we have introduced the notion of polynomial growth of some degree of a function with respect to a real function and proved that any non-negative twice differentiable subharmonic functions in an $n$-dimensional manifold always admit polynomial growth of degree $1$ with respect to a non-negative real valued subharmonic function on real line. We have also given a lower bound of the integration of a convex function in a geodesic ball. 
\end{abstract}
\noindent\footnotetext{ $^*$ Corresponding author.\\
$\mathbf{2010}$\hspace{5pt}Mathematics\; Subject\; Classification: 53C21, 53C43, 58E20,  58J05. }

\section{introduction}
A real valued function $f$ is said to have polynomial growth of order $q\in \mathbb{R}$ if there is a constant $C>0$ such that $|f|(x)\leq Cr^q(x)$ for all $x\in M$, where $r(x)$ is the distance of $x\in M$ from a fixed point $x_0\in M$ and is denoted by $|f(x)|=O(r^q)$. The growth of harmonic and subharmonic functions have been studied by many authors \cite{LT89}, \cite{YAU75}. The growth of a function depends on the geometrical structure of the manifold. The first significant work about the harmonic function has been done by Yau \cite{YAU75} in 1975. A function $f$ is said to have sublinear growth if $|f(x)|=o(r)$. Cheng \cite{CH80} proved that a complete Riemannian manifold with non-negative Ricci curvature does not admit any non-constant harmonic function with sublinear growth. This led Yau to formulate the following conjecture. 
\begin{conjecture}
For each integer $q$, the space of harmonic functions on a manifold $M$ with non-negative Ricci curvature satisfying
$|f(x)|=O(r^q(x)),$ is finite dimensional.
\end{conjecture}
Li and Tam \cite{LT89} proved that if the volume function of geodesic balls has polynomial growth of order $q>0$, then the space of harmonic functions with polynomial growth of degree $1$ is finite dimensional.
In this article we have generalized the notion of polynomial growth of a real valued function in a Riemannian manifold. We have defined polynomial growth of a function with respect to a real function ( If a function is from $\mathbb{R}$ to $\mathbb{R}$, then we call it a real function) and showed that under suitable conditions all non-negative twice differentiable subharmonic functions have polynomial growth of degree $1$ with respect to a real subharmonic function. \\
\indent We organize this paper as follows: Section 2 deals with some preliminaries of Riemannian manifold. In this section we have introduced the notion of polynomial growth with respect to a function, which is the natural generalization of polynomial growth of a function in a Riemannian manifold. In section 3 we have proved that every non-negative $C^2$ subharmonic function in a strongly symmetric Riemannian manifold with a pole possesses polynomial growth of degree $1$ with respect to a non-negative real subharmonic function. In this section we have also given a lower bound for the integration of a convex function in terms of volume, distance function and a real subharmonic function.
\section{Preliminaries} 
In this section we have discussed some basic facts of a Riemannian manifold $(M,g)$, which will be used throughout this paper (for reference see \cite{PE06}). Throughout this paper by $M$ we mean a complete Riemannian manifold of dimension $n$ endowed with some positive definite metric $g$ unless otherwise stated. The tangent space at the point $p\in M$ is denoted by $T_pM$ and the tangent bundle is defined by $TM=\cup_{p\in M}T_pM$. The length $l(\gamma)$ of the curve $\gamma:[a,b]\rightarrow M$ is given by
\begin{eqnarray*}
l(\gamma)&=&\int_{a}^{b}\sqrt{g_{\gamma(t)}(\dot{\gamma}(t),\dot{\gamma}(t))}\ dt\\
&=&\int_{a}^{b}\norm {\dot{\gamma}(t)}dt.
\end{eqnarray*}
The curve $\gamma$ is said to be a geodesic if $\nabla_{\dot{\gamma}(t)}\dot{\gamma}(t)=0\ \forall t\in [a,b]$, where $\nabla$ is the Riemannian connection of $g$.
For any point $p\in M$, the exponential map $exp_p:V_p\rightarrow M$ is defined by
$$exp_p(u)=\sigma_u(1),$$
where $\sigma_u$ is a geodesic with $\sigma(0)=p$ and $\dot{\sigma}_u(0)=u$ and $V_p$ is a collection of vectors of $T_pM$ such that for each element $u\in V_p$, the geodesic with initial tangent vector $u$ is defined on $[0,1]$. It can be easily seen that for a geodesic $\sigma$, the norm of a tangent vector is constant, i.e., $\norm {\dot{\gamma}(t)}$ is constant. If the tangent vector of a geodesic is of unit norm, then the geodesic is called normal. If the exponential map exp is defined at all points of $T_pM$ for each $p\in M$, then $M$ is called complete. Hopf-Rinow theorem provides some equivalent cases for the completeness of $M$. Let $x$, $y\in M$. The distance between $p$ and $q$ is defined by
$$d(x,y)=\inf\{l(\gamma):\gamma \text{ be a curve joining }x \text{ and }y\}.$$
A geodesic $\sigma$ joining $x$ and $y$ is called minimal if $l(\sigma)=d(x,y)$. Hopf-Rinow theorem guarantees the existence of minimal geodesic between two points of $M$. A smooth vector field is a smooth function $X:M\rightarrow TM$ such that $\pi\circ X=id_M$, where $\pi:TM\rightarrow M$ is the projection map.\\
\indent A pole in $M$ is such a point where the tangent space is diffeomorphic to the whole manifold. Gromoll and Meyer \cite{GM69} introduced the notion of pole in $M$. A point $o\in M$ is called a pole of $M$ if the exponential map at $o$ is a global diffeomorphism and a manifold $M$ with a pole $o$ is denoted by $(M,o)$. If a manifold possesses a pole then the manifold is complete. Simply connected complete Riemannian manifold with non-positive sectional curvature and a paraboloid of revolution are the examples of Riemannian manifolds with a pole. A manifold with a pole is diffeomorphic to the Euclidean space but the converse is not always true, see \cite{14}.\\
\indent  The gradient of a smooth function $f:M\rightarrow\mathbb{R}$ at the point $p\in M$ is defined by $\nabla f(p)=g^{ij}\frac{\partial f}{\partial x_{j}}\frac{\partial}{\partial x_i}\mid_p.$ It is the unique vector field such that $g( \nabla f,X)=X(f)$ for all smooth vector field $X$ in $M$. The Hessian $Hess(f)$ is the $(0,2)$-tensor field, defined by $Hess(f)(X,Y)=g(\nabla_X\nabla f,Y)$ for all smooth vector fields $X,Y$ of $M$. For a vector field $X$, the divergence of $X$ is defined by
$$div(X)=\frac{1}{\sqrt{g}}\frac{\partial }{\partial x^j}\sqrt{g}X^j,$$
where $g=det(g_{ij})$ and $X=X^j\frac{\partial}{\partial x^j}$. The Laplacian of $f$ is defined by $\Delta f=div(\nabla f)$. 
\begin{definition}\cite{YAU75}
A $C^2$-function $u:M\rightarrow\mathbb{R}$ is said to be harmonic if $\Delta u=0$. The function $u$ is called subharmonic (superharmonic) if $\Delta \geq (\leq )0$.
\end{definition}
\begin{definition}\cite{UDR94}\cite{RAP97}
 A real valued function $f$ on $M$ is called convex if for every geodesic $\gamma:[a,b]\rightarrow M$, the following inequality holds
\begin{equation*}
f\circ\gamma((1-t)a+tb)\leq (1-t)f\circ\gamma(a)+tf\circ\gamma(b)\quad \forall t\in [0,1],
\end{equation*}
or if $f$ is differentiable, then 
\begin{equation*}
g(\nabla f,X)_x\leq f(exp_x\nabla f)-f(x), \ \forall X\in T_xM.
\end{equation*}
\end{definition}
If the function $f$ is $C^2$, then the convexity of $f$ is equivalent to $Hess(f)\geq 0$. Zhang and Xu \cite{ZX04} introduced the notion of strongly symmetric manifold. And they studied some properties of subharmonic function in strongly symmetric manifold with a pole. Some discussion about strongly symmetric manifold can also be found in \cite{GW79}, where they use the term ``model'' instead of `` strongly symmetric manifold''. 
\begin{definition}\cite{ZX04}
A manifold $(M,o)$ with a pole $o$ is strongly symmetric around $o$ if and only if every linear isometry $\phi_o:T_oM\rightarrow T_oM$ can be realized as the differential of an isometry $\phi:M\rightarrow M$, i.e., $\phi(o)=o$ and $d\phi(o)=\phi_o$.
\end{definition}

\begin{definition}
A function $u:M\rightarrow\mathbb{R}$ is said to have \textit{polynomial growth of degree $q\in\mathbb{R}$ with respect to $v$}, where $v$ is a real function, if $|f(x)|=O(r^qv(r)),$ where $r(x)$ is the distance of $x$ from a fixed point in $M$.
\end{definition}
We say that a function has $(v,p)$-polynomial growth if it has polynomial growth of degree $p$ with respect to $v$. If $v\equiv 1$, then we get the definition of polynomial growth. Then using this generalized notion of polynomial growth, we can generalized the Conjecture 1.1 and ask the following:
\begin{conjecture}
 For a fixed integer $p$ and real function $v$, the space of harmonic functions with $(v,p)$-polynomial growth on a Riemannian manifold having non-negative Ricci curvature is finite dimensional?
\end{conjecture}
\section{polynomial growth and subharmonic functions}
\begin{theorem}\cite{LS84}
Let $M$ be an $n$-dimensional complete manifold with non-negative Ricci curvature. Then any non-negative subharmonic function $u:M\rightarrow\mathbb{R}$ satisfies
\begin{equation}\label{eq4}
\sup_{B_{r/2}}u\leq  \frac{C}{Vol(B_r)}\int_{B_r}u,
\end{equation}
where the constant $C$ depends only on $n$.
\end{theorem}
\begin{theorem}\label{th1}
Let $(M,o)$ be an $n$-dimensional manifold with a pole $o$ and $Ric_M\geq 0$. If $M$ is strongly symmetric around $o$, then for every non-negative subharmonic function $u\in C^2(M)$, there exists a non-negative subharmonic function $v$ on $\mathbb{R}$ such that.
\begin{equation}\label{eq5}
\sup_{B_{r/2}}u\leq rCv(r) \text{ for every }r>0,
\end{equation}
where $C$ is the constant, depends only on $n$. In particular, $u(x)=O(r\omega(r)),$ for some non-negative subharmonic function $\omega$ in $\mathbb{R}$.
\end{theorem}
\begin{proof}
Since $M$ is diffeomorphic to the euclidean space so by taking the polar coordinates of $\mathbb{R}^n$ as $(r,\theta^1,\cdots,\theta^{n-1})$, the metric of $M$ can be expressed in polar from, i.e.,
\begin{equation}\label{eq12}
ds^2=dr^2+\sum_{i,j}g_{ij}d\theta^id\theta^j=dr^2+h(r)^2d\Theta^2
\end{equation}
on $M-\{o\}$, where $g_{ij}=g(\frac{\partial}{\partial \theta^i},\frac{\partial}{\partial \theta^j})$ and $d\Theta^2$ is the canonical metric on the unit sphere of $T_oM$. Since $M$ is strongly symmetric around $o$ so $h$ depends only on $r$. Hence $r(x)$ denotes the geodesic distance from $o$ to $x$ for any $x\in M$. The Riemannian volume element of $S_r$ can be expressed as $dS_r=\sqrt{D(r,\Theta)}d\theta^1\cdots d\theta^{n-1},$ where $D=det(g_{ij})$.
Since $u$ is subharmonic so from (\ref{eq4}), we get
\begin{equation*}
\sup_{B_{r/2}}u\leq  \frac{C}{Vol(B_r)}\int_{B_r}udV.
\end{equation*}
Since for $r>0$, $Vol(\partial B_r)\leq Vol(B_{r})$, which implies that $\frac{1}{Vol(B_r)}\leq \int_{\epsilon}^{r}\frac{1}{Vol(\partial B_t)}dt,$ for some arbitrary small $\epsilon>0$ and hence we get
\begin{eqnarray*}
\sup_{B_{r/2}}u&\leq & \frac{C}{Vol(B_r)}\int_{0}^{r}\int_{\partial B_t}udS_tdt\\
&\leq & \liminf_{\epsilon\rightarrow 0} C\int_{\epsilon}^{r}\frac{1}{Vol(\partial B_t)}\int_{\partial B_t}udS_tdt
\end{eqnarray*}
Now for $r>0$ define 
\begin{equation}\label{eq1}
v(r)=\frac{1}{Vol(\partial B_r)}\int_{\partial B_r}udS_r.
\end{equation}
 Then we get
\begin{equation}\label{eq6}
\sup_{B_{r/2}}u\leq C\liminf_{\epsilon\rightarrow 0}\int_{0}^{r}v(t)dt.
\end{equation}
 Using the coordinate system (\ref{eq12}), the equation (\ref{eq1}) can be represented as 
\begin{equation*}
v(r)=\frac{1}{Vol(\partial B_1)}\int_{\partial B_1}u(r\xi)dS_1, \text{ for }\xi\in \partial B_1.
\end{equation*}
Now taking derivative with respect to $r$ we get
\begin{equation*}
v'(r)=\frac{1}{Vol(\partial B_1)}\int_{\partial B_1}\partial_ru(r\xi)dS_1=\frac{1}{Vol(\partial B_r)}\int_{\partial B_r}\partial_rudS_r.
\end{equation*}
Hence by using divergence theorem, we obtain
\begin{equation}\label{eq3}
v'(r)Vol(\partial B_r)=\int_{\partial B_r}\partial_rudS_r= \int_{B_r}\Delta udV,
\end{equation}
where $dV$ is the volume element of $M$. Now it can be easily seen that $v$ is radially symmetric function. Hence 
\begin{equation}\label{eq2}
\Delta v=v''+(\Delta r)v'.
\end{equation}
Again we have the following equation
$$\int_{B_r}\Delta udV=\int_{0}^{r}\int_{\partial B_t}\Delta udS_tdt.$$
Also in \cite{ZX04} it is proved that $\Delta r=\frac{Vol'(\partial B_r)}{Vol(\partial B_r)}$, hence using this relation and (\ref{eq1}) and (\ref{eq2}) we get
\begin{eqnarray*}
\int_{\partial B_r}\Delta udS_r&=&\frac{d}{dr}\int_{B_r}\Delta udV=\frac{d}{dr}[v'(r)Vol(\partial B_r)]\\
&=& Vol(\partial B_r)\Big[v''(r)+\frac{Vol'(\partial B_r)}{Vol(\partial B_r)}v'(r)\Big]\\ &=& Vol(\partial B_r)\Big[v''(r)+(\Delta r)v'(r)\Big]\\
&=& Vol(\partial B_r)\Delta v(r).
\end{eqnarray*}
Thus we obtain
\begin{equation}
\Delta v(r)=\frac{1}{Vol(\partial B_r)}\int_{\partial B_r}\Delta udS_r.
\end{equation}
Since $M$ is strongly symmetric, hence \cite[Lemma 3.1]{ZX04} $\lim\limits_{r\rightarrow 0}r\Delta r=n-1$. Then for $r=0$, we have
$v'(0)=0,\quad v''(0)=\frac{1}{n}\Delta u(0),\quad \lim\limits_{r\rightarrow 0}\Delta v(r)=\Delta u(0).$
Then $v\in C^2(\mathbb{R})$ \cite{ZX04}. Now $\Delta u\geq 0$, hence for any $r\geq 0$, the above inequality implies that $\Delta v\geq 0$, i.e., $v$ is subharmonic.
Then (\ref{eq6}) implies that
\begin{eqnarray}
\sup_{B_{r/2}}u&\leq & C\int_{0}^{r}v(t)\quad (\text{since }v \text{ is continuous at }0)\\
& \leq & rC\sup_{[0,r]}v \text{ for any }r>0.
\end{eqnarray}
Now $v$ is subharmonic in $[0,R]$, so using maximum principle $\sup_{[0,R]}v=v(R)$, since $v$ is non-decreasing \cite{ZX04}. Hence we get
\begin{equation*}
\sup_{B_{r/2}}u\leq rCv(r) \text{ for any }r>0.
\end{equation*}
The second part is proved trivially from the first part and by taking $\omega(r)=v(2r)$.
\end{proof}
\begin{corollary}
Let $u\in C^2(\mathbb{R}^2)$ be a subharmonic function. Then $u(x)=O(r\omega(r))$, for some non-negative subharmonic function $\omega$ in $\mathbb{R}$.  
\end{corollary}
\begin{proof}
Since there is no non-constant negative subharmonic function in $\mathbb{R}^2$, so the proof easily follows from the above Theorem.
\end{proof}
\begin{theorem}\cite{GW74}
Let $M$ be a complete noncompact Riemannian manifold
of positive sectional curvature. If $u$ is a continuous nonnegative subharmonic function on $M$, then for any $p > 1$ there exist positive constants $C$ and $r_0$ such that 
\begin{equation}\label{eq7}
\int_{B_r}u^pdV\geq C(r-r_0)\quad\text{ for all }r\geq r_0.
\end{equation}
\end{theorem}
\begin{theorem}
Let $(M,o)$ be an $n$-dimensional manifold with a pole $o$ and of positive sectional curvature. Then for every non-negative convex function $u$ on $M$, there exist constants $C>0$ and $r_1>0$ such that 
\begin{equation*}
u(o)\geq \frac{2C}{vol(\partial B_{r_1})}-\sup_{\partial B_{r_1}}u.
\end{equation*}
\end{theorem}
\begin{proof}
Since $u$ is convex, so $u$ is also subharmonic \cite{GW71}. Hence from (\ref{eq7}) we get
\begin{equation*}
 C(r-r_0)\leq \int_{B_r}u^pdV\leq \int_{\partial B_r}\int_{0}^{1}u(\sigma_x(t))dtdS_r\quad\text{ for all }r\geq r_0,
\end{equation*}
where $\sigma_x:[0,1]\rightarrow M$ is the minimal geodesic such that $\sigma_x(0)=o$ and $\sigma_x(1)=x$.
Now using convexity of $u$ and for $r>r_0$, we obtain
\begin{eqnarray*}
C(r-r_0)&\leq &\int_{\partial B_r}\int_{0}^{1}u(\sigma_x(t))dtdS_r\\
&\leq &\int_{\partial B_r}\int_{0}^{1}[(1-t)u(o)+tu(x)]dtdS_r\\
&\leq &\frac{1}{2}\int_{\partial B_r}[u(o)+u(x)]dS_r\\
\frac{2C(r-r_0)}{vol(\partial B_r)}&\leq & u(o)+\sup_{\partial B_r}u.
\end{eqnarray*}
Now taking $r=R_0+1$, we get
\begin{equation*}
u(o)\geq \frac{2C}{vol(\partial B_{r_0+1})}-\sup_{\partial B_{r_0+1}}u.
\end{equation*}
\end{proof}
Let $u$ be a non-negative subharmonic function. Then in $B_{2R}$ we have \cite[p. 78]{SY94}
$$\int_{B_r}|\nabla u|^2dV\leq \frac{C}{r^2}\int_{B_{2r}}u^2dV\leq Vol(B_{2r})\frac{C}{r^2}\sup_{B_{2r}}u^2.$$
Since $u$ is subharmonic so $u^2$ is also subharmonic. Hence by applying the Theorem \ref{th1}, we obtain
$$\int_{B_r}|\nabla u|^2dV\leq Vol(B_{2r})\frac{C}{r^2}rv(4r)=Vol(B_{2r})\frac{C}{r}v(4r), $$
for some non-negative subharmonic function $v$ in $[0,4r]$. Hence 
\begin{equation}\label{eq8}
\frac{r}{Vol(B_{2r})}\int_{B_r}|\nabla u|^2dV\leq Cv(4r).
\end{equation}
 Now taking limit, we get
 \begin{equation*}
 \limsup_{r\rightarrow\infty}\frac{r}{Vol(B_{2r})}\int_{B_r}|\nabla u|^2dV\leq C\limsup_{r\rightarrow\infty}v(4r).
 \end{equation*}
 Since $Ric_M\geq 0$, so by Bishop volume comparison Theorem \cite[p. 11]{SY94}, $Vol(B_r)\leq C_nr^n$. So from the above inequality we get
  \begin{equation*}
  \limsup_{r\rightarrow\infty}\frac{1}{r^{n-1}}\int_{B_r}|\nabla u|^2dV\leq C_1\limsup_{R\rightarrow\infty}v(4r),
  \end{equation*}
  for some constant $C_1$, depends only on $n$. Hence we sate the following Proposition:
  \begin{proposition}
 Under the assumption of Theorem \ref{th1}, there exists a non-negative subharmonic function $v$ in $\mathbb{R}$ such that, for all $r>0$
  \begin{equation}
   \limsup_{r\rightarrow\infty}\frac{1}{r^{n-1}}\int_{B_r}|\nabla u|^2dV\leq C_1\limsup_{r\rightarrow\infty}v(4r).
   \end{equation}
  \end{proposition}
    \begin{theorem}
   Let $(M,o)$ be an $n$-dimensional manifold with a pole $o$ and $Ric_M\geq 0$. If $M$ is strongly symmetric around $o$, then for every  convex function $u\in C^2(M)$ with $u\geq 1$ and $|\nabla u|\geq 1$, there exists a positive subharmonic function $v$ in $\mathbb{R}$ such that, for all $r>0$ 
    \begin{equation*}
     \int_{B_r}u(exp_x\nabla u) dV\geq C_6(n) \frac{(Vol(B_r))^2}{r^{n+1}v^3(4r)}, 
     \end{equation*}
     where $C_6>0$ is a constant, depends only on $n$.
    \end{theorem}
 \begin{proof}
 The convexity of $u$ and (\ref{eq4}) imply that 
 \begin{equation*}
 \sup_{\partial B_{r/2}}u\leq \frac{C}{Vol(B_r)}\int_{B_r}udV,\quad \text{ for all }r>0.
 \end{equation*}
 Now applying Schwarz's Inequality, we obtain
 \begin{equation}\label{eq9}
 \sup_{\partial B_{r/2}}u\leq \frac{C}{Vol(B_r)}\Big(\int_{B_r}\frac{u^2}{|\nabla u|^2} dV\Big)^{1/2}\Big(\int_{B_r}|\nabla u|^2 dV\Big)^{1/2}.
 \end{equation}
 Now from (\ref{eq8})
 $$\int_{B_r} |\nabla u|^2\leq C_1\frac{Vol(B_{2r})}{r}v(4r),\quad
 \forall r>0,$$
 where $v$ is a positive subharmonic function in $\mathbb{R}$ and $C_1>0$ is a constant depends only on $n$. Hence putting this value in (\ref{eq9}), we get
 \begin{equation*}
  \sup_{\partial B_{r/2}}u\leq \frac{C_3}{Vol(B_r)}\Big(\int_{B_r}\frac{u^2}{|\nabla u|^2} dV\Big)^{1/2}\Big(\frac{Vol(B_{2r})}{r}v(4r)\Big)^{1/2},
 \end{equation*}
 for some constant $C_3>0$, depends only on $n$.
 Now by Bishop volume comparison $Vol(B_{2r})\leq C_n2^nr^n$, we get
  \begin{equation*}
   \sup_{\partial B_{r/2}}u\leq \frac{C_3}{Vol(B_r)}\Big(\int_{B_r}\frac{u^2}{|\nabla u|^2} dV\Big)^{1/2}(C_n2^nr^{n-1}v(4r))^{1/2},
  \end{equation*}
  i.e.,
   \begin{equation*}
    (\sup_{\partial B_{r/2}}u)^2\leq \frac{C_4(n)}{(Vol(B_r))^2}\Big(\int_{B_r}\frac{u^2}{|\nabla u|^2} dV\Big)r^{n-1}v(4r),
     \end{equation*} 
    for some constant $C_4>0$. From (\ref{eq5}), we get $rCv(2r)\geq \sup_{B_r}u\geq u(x)\ \forall x\in B_r,$ i.e., $r^2C^2v^2(2r)\geq u^2(x)\ \forall x\in B_r.$ Hence $\sup_{B_r}u^2\leq r^2C^2v^2(2r)$. Now rewriting the above inequality we have
      \begin{eqnarray*}
        1&\leq & \frac{C_4(n)}{(Vol(B_r))^2}\Big(\int_{B_r}\frac{u^2}{|\nabla u|^2} dV\Big)r^{n-1}v(4r)\\
        &\leq & \frac{C_4(n)}{(Vol(B_r))^2}\Big(\sup_{B_r}u^2\int_{B_r}\frac{1}{|\nabla u|^2} dV\Big)r^{n-1}v(4r)\\
        &\leq & r^2C^2v^2(2r)\frac{C_4(n)}{(Vol(B_r))^2}\Big(\int_{B_r}\frac{1}{|\nabla u|^2} dV\Big)r^{n-1}v(4r).
         \end{eqnarray*} 
   Since $v$ is non-decreasing so $v(2r)\leq v(4r)$ for $r>0$. Hence the above inequality implies that
   \begin{equation*}
   \frac{C_5(n)}{(Vol(B_r))^2}\Big(\int_{B_r}\frac{1}{|\nabla u|^2} dV\Big)r^{n+1}v^3(4r)\geq 1,
   \end{equation*}
   i.e.,
  \begin{equation}\label{eq10}
  \int_{B_r}\frac{1}{|\nabla u|^2} dV\geq \frac{(Vol(B_r))^2}{C_5(n)r^{n+1}v^3(4r)}
  \end{equation}
   for some constant $C_5>0$. Again $|\nabla u|>1$, so $|\nabla u|^2\geq \frac{1}{|\nabla u|^2}$. Hence (\ref{eq10}) implies that
   \begin{equation}\label{eq11}
     \int_{B_r}|\nabla u|^2 dV\geq \frac{(Vol(B_r))^2}{C_5(n)r^{n+1}v^3(4r)}.
     \end{equation}
     Now $u$ is convex and $u>0$, so we get
     $$|\nabla u|^2_x\leq u(exp_x\nabla u)-u(x)\leq u(exp_x\nabla u) \ \forall x\in B_r.$$
     Hence by taking $C_6=1/C_5$, (\ref{eq11}) implies
    \begin{equation*}
        \int_{B_r}u(exp_x\nabla u) dV\geq C_6(n) \frac{(Vol(B_r))^2}{r^{n+1}v^3(4r)}.
        \end{equation*}  
       And the positivity of $v$ can easily be seen by using $(\ref{eq1})$ and $u>0$.
 \end{proof}
\section*{Acknowledgment}
 The second author greatly acknowledges to
The University Grants Commission, Government of India for the award of Junior Research
Fellowship.

$\bigskip $

$^{1,2}$Department of Mathematics, 

The University of
 Burdwan, Golapbag, Burdwan-713104,
 
 West Bengal, India.
 
$^1$E-mail:aask2003@yahoo.co.in, aashaikh@math.buruniv.ac.in

$^2$E-mail:chan.alge@gmail.com
$\bigskip $

\end{document}